\documentclass{amsart}
\usepackage{enumerate}
\usepackage{amssymb}
\usepackage{mathtools}
\usepackage{mathrsfs}
\usepackage{color}
\usepackage[
colorlinks=true,
linkcolor=blue,
citecolor=blue,
urlcolor=blue]{hyperref}

\numberwithin{equation}{section}
\newcommand{\E}{\mathbb{E}}
\newcommand{\M}{\mathcal{M}}
\newcommand{\cn}{\Omega_{n}}
\newcommand{\MM}{\mathbb{M}}
\newcommand{\x}{\mathbf{x}}
\newcommand{\y}{\mathbf{y}}

\theoremstyle{plain}
\newtheorem{theorem}{Theorem}[section]

\newtheorem{corollary}[theorem]{Corollary}
\theoremstyle{remark}
\newtheorem{remark}[theorem]{Remark}

\begin{document}
\title{Talagrand type for noncommutative $L_1$ spaces}
	\author[S.Luo]{Sijie Luo}
	\address{
		School of Mathematics and Statistics\\
		Central South University\\
		Changsha 410083\\
		China}
	
	\email{sijieluo@csu.edu.cn}

	\author[Y.Wang]{Yusen Wang${}^{\ast}$}
	\address{
		School of Mathematics and Statistics\\
		Central South University\\
		Changsha 410083\\
		China}
	
	\email{yusenwang24@csu.edu.cn}

	\author[D.Zhou]{Dejian Zhou}
	\address{
		School of Mathematics and Statistics\\
		Central South University\\
		Changsha 410083\\
		China}
	
	\email{zhoudejian@csu.edu.cn}

\thanks{$\ast$ the corresponding author.}
	\maketitle
	
	\begin{abstract}
		Let $(\mathcal{M},\tau)$ be a semi-finite von Neumann algebra. We prove that $L_1(\mathcal{M})$ has Talagrand type $(1,\psi_{1,1})$, where $\psi_{1,1}(t)=t/\log(e+t)$. As an consequence, the Schatten trace class $S_1$ has Talagrand type $(1,\psi_{1,1})$, thus settled an open problem of Cordero-Erausquin and Eskenazis \cite[Question 2]{CE2023}. Our proof is based on the noncommutative Mazur maps initiated by Ricard \cite{R2015}.
	\end{abstract}
	
	\section{introduction}
	
	Motivated by the developments of nonlinear geometry of Banach spaces, where the classification of Banach spaces is an significant researching direction, the study of vector-valued functional inequalities form an essential theme in this area. To put the development of this topic in context, we quote two well-known results. First, a surjective isometry between two Banach spaces which takes $0$ to $0$ is necessarily linear which is given by Mazur and Ulam \cite{MU1932}. It means that the structure of a metric space is completely determined its linear structure. The second is according to Kadec \cite{Ka1967}, is that all separable infinite-dimensional Banach spaces are homeomorphic with each other.  The second says that its structure as a topological space, which is very simple, contains no information whatsoever on its linear structure. Based on these consequences, one of the basic questions in this area is whether linear properties of Banach spaces can be characterized in terms of their metric structure. Enflo's work on uniform homeomorphisms \cite{Enflo1969, Enflo1978} introduces an influential tool we call the Enflo type today and asked whether Rdemacher type implies with it the same $p$. Specifically, a Banach space has Enflo type expresses that the length of diagonals can be controlled by the sum of lengths of vertices. In the probabilistic theory of Banach spaces another essential vector-valued inequality is established by Pisier \cite{Pi1986} on the discrete cube and denoted by Pisier's inequality which is closely related to what we mentioned above. Pisier proved that Rademacher type $p$ implies Enflo type $r$ for every $1\leq r<p$ by it. This inequality given by Pisier in 1986 has coefficients with a logarithmic dependence on the dimension. Then in 2020, Hyt\"onen, Tuomas and Naor, Assaf \cite{HN2013} improved it to a universal constant. Ivanisvili, van Handel and Volberg \cite{IVHV2020} resolved Enflo that it is equivalent to Rademacher type by using the dimension-free Pisier's inequality. The question of whether these two notions coincide has been an important motivation for the study of vector-valued Poincar\'e inequalities. In particular, their results show that the discrete Poincar\'e inequality with $L_2$ norms holds for Banach-space-valued functions when the target space has Rademacher type $2$.
	
%A related development comes from the analysis of Boolean functions. The theorem of Kahn, Kalai and Linial \cite{KKL1988} gives a logarithmic improvement over the lower bound on the largest coordinate influence obtained from the Poincar\'e inequality.  Ivanisvili and Stone \cite{IS2024} proved a vector-valued KKL inequality for every Banach space of Rademacher type $2$. Their result further illustrates the connection between the geometry of the target space and logarithmic refinements of functional inequalities on the discrete cube.

Talagrand's influence inequality \cite{Ta1994}, also known as the $L_1$--$L_2$ inequality, refines the scalar Poincar\'e inequality by introducing a logarithmic denominator for each coordinate derivative. The improvement depends on the ratio of the $L_2$-norm to the $L_1$-norm of that derivative. In particular, Talagrand's inequality implies the Kahn--Kalai--Linial theorem. This leads naturally to the question of whether the same logarithmic improvement holds for Banach-space-valued functions. Cordero-Erausquin and Eskenazis \cite{CE2023} obtained a vector-valued extension under the assumption of Rademacher type $2$, with an additional doubly logarithmic factor, and removed this loss under the stronger assumption of martingale type $2$. They also proved stronger formulations in terms of Orlicz norms. The inequality of Ivanisvili and Stone has a common logarithmic denominator and does not directly yield Talagrand's inequality with a separate denominator for each coordinate.

To study the corresponding inequalities in metric spaces, Cordero-Erausquin and Eskenazis introduced Talagrand type. Their result \cite[Theorem~42]{CE2023} showed that scale-valued $L_1(\mu)$ space has Talagrand type $(1,\psi_{1,1})$, where $\psi_{1,1}(t)=t/\log(e+t)$. Their proof uses Schoenberg's theorem to construct an isometry from $(\ell_1, \| \cdot \|_{\ell_1}$ to $\left(\ell_2(\ell_2), \| \cdot \|^2_{\ell_2(\ell_2)}\right)$. This argument also used the metric structure of classical $L_1$ spaces. They therefore asked whether the Schatten trace class $S_1$ has the same Talagrand type \cite[Question~2]{CE2023}. Since $S_1$ contains an isometric copy of $\ell_1$, the results for spaces of nontrivial Rademacher type do not directly answer this question.

In this paper, we prove the corresponding inequality for noncommutative $L_1$ spaces associated with arbitrary semifinite tracial von Neumann algebras. In particular, it gives an affirmative answer to Question~2 of \cite{CE2023}.

\begin{theorem}\label{main}
Let $(\M,\tau)$ be a von Neumann algebra equipped with a faithful normal semifinite trace. There exists a universal constant $C<\infty$ such that, for every $n\in\mathbb N$ and every $f:\cn\to L_1(\M,\tau)$,
\begin{equation}\label{eqmain}
\E_{\varepsilon,\delta}
\|f(\varepsilon)-f(\delta)\|_{L_1(\M)}
\leq C\sum_{j=1}^n
\bigl\|\,\|\partial_j f\|_{L_1(\M)}\,\bigr\|_{\psi_{1,1}}.
\end{equation}
Consequently, $L_1(\M,\tau)$ has Talagrand type $(1,\psi_{1,1})$, with a constant independent of the algebra. In particular, the Schatten trace class $S_1$ has Talagrand type $(1,\psi_{1,1})$.
\end{theorem}

By Jensen's inequality, we also obtain the following centered formulation.
\begin{corollary}[Centered formulation]\label{cor:centered}
Under the assumptions of Theorem~\ref{main},
\[
\E\|f-\E f\|_{L_1(\M)}
\leq C\sum_{j=1}^n
\bigl\|\,\|\partial_j f\|_{L_1(\M)}\,\bigr\|_{\psi_{1,1}}.
\]
\end{corollary}

Our proof combines Ricard's estimates for the noncommutative Mazur maps \cite{R2015} with the Hilbert-valued case of \cite[Theorem~7]{CE2023}. After centering the original function, we apply the Mazur map from $L_1(\M)$ to $L_2(\M)$. The estimate for the inverse map, together with the Cauchy--Schwarz inequality, bounds the average $L_1$ distance by the variance of the transformed function. We then apply the Hilbert-valued Talagrand inequality. Finally, the $1/2$-H\"older estimate for the forward Mazur map controls the discrete derivatives, and the comparison of the corresponding Luxemburg functionals gives the right-hand side of \eqref{eqmain}.

The paper is organized as follows. Section~2 introduces the notation and recalls the required facts about noncommutative $L_p$ spaces, Orlicz norms and Mazur maps. Section~3 establishes the Mazur map estimates needed in the proof and records the Hilbert-valued Talagrand inequality. Section~4 contains the proof of Theorem~\ref{main}.

	\section{Preliminaries}\label{pre}
	Let $\cn=\{-1,1\}^{n}$
	with the uniform probability measure $\sigma_n$. If $\varepsilon=(\varepsilon_1,\ldots,\varepsilon_n)\in\cn$, write
	\[
	\varepsilon^{(j)}=(\varepsilon_1,\ldots,\varepsilon_{j-1},-\varepsilon_j,\varepsilon_{j+1},\ldots,\varepsilon_n)
	\]
	and, for a function $f:\cn\to X$ with values in a vector space,
	\[
	\partial_{j} f(\varepsilon)=\frac{f(\varepsilon)-f(\varepsilon^{(j)})}{2}.
	\]
	In what follows, $\varepsilon$ and $\delta$ are assumed to be elements of $\cn$. We set $\sigma_{2n}=\sigma_n\otimes\sigma_n$ on $\cn\times\cn$. Expectations in both variables are taken with respect to this product measure.
	
	Let $(\M,\tau)$ be a von Neumann algebra equipped with a faithful normal semifinite trace. We use the standard tracial noncommutative spaces $L_p(\M,\tau), p \geq 1$, abbreviated to $L_p(\M)$. Their norms are denoted by $\|\cdot\|_{L_p(\M)}$. It must be mentioned that every $x\in L_p(\M)$ has a polar decomposition $x=u|x|, u\in\mathcal M$ in the algebra of $\tau$-measurable operators affiliated with $\M$. We use the shorthand notation $a\lesssim b$ and $a\gtrsim b$ to mean $a\leq Cb$ and $b\leq Ca$ for a constant $C$. Furthermore, $a \asymp b$ imply both $a \lesssim b$ and $a \gtrsim b$.

	\subsection{Orlicz norm}
	
	Recall that with a Young function $\phi$ (i.e. a convex function with $\lim_{x\to 0}\frac{\phi(x)}{x}=0$ and $\lim_{x\to\infty}\frac{\phi(x)}{x}=\infty$) and $E$ is a Banach space, the $\phi$-Orlicz norm $\| \cdot \|_{L_{\phi}}$ of $f: \cn \to E$ is defined by
	\begin{equation}\label{Orlicznorm}
		{\left\| f \right\|}_{L_{\phi}} := \inf \left\{ \lambda > 0: \int_{\cn} \phi\left( \frac{\|f(\varepsilon)\|}{\lambda} \right) d \sigma_{n}(\varepsilon) \leq 1 \right\}.
	\end{equation}
	We denote the norm by $\| \cdot \|_{{L_{p}(\log L)}^{-r}}$
	when $p\in(1,\infty)$, $r\in\mathbb R$, and $\phi$ is a Young
	function satisfying $\phi(t)=t^p\log^{-r}(e+t)$ for all sufficiently
	large $t$. The space ${L_p(\log L)}^{-r}$ is the complete space
	under the corresponding norm.
	
	Now we introduce a very important Young function $\Phi: [0,\infty) \to [0, \infty)$ :
	\[\Phi(t) = 2 \int_{0}^{t} \frac{s}{\log(e+s)} ds = 2 \int_{0}^{t} \psi_{1,1}(s) ds,\qquad t \geq 0.\]
	It is straightforward that $\Phi(t)$ is an increasing function. Since $\psi_{1,1}$ is increasing, 
	\begin{equation}
		\frac{s}{\log(e+t)} \leq \frac{s}{\log(e+s)} \leq \frac{t}{\log(e+t)}, \qquad 0 \leq s \leq t.
	\end{equation}
	Hence we can obtain
	\[\Phi(t) \asymp \frac{t^2}{\log(e+t)},\quad t\ge 0,\]
	and 
	\begin{equation}\label{normequvalience1}
		\| \cdot \|_{L_{\Phi}} \asymp \| \cdot \|_{{L_{2}(\log L)}^{-1}}.
	\end{equation}
	Having established the foregoing definitions and relations, we proceed to define the increasing function
	\[\Lambda(t) = \Phi({t}^{1/2}).\] 
	It is obvious that
	\[\Lambda(t)\asymp\frac{t}{\log(e+t)},\qquad t\geq0,\]
	%	 \[	\|\cdot\|_{L_{\Lambda}} \asymp \| \cdot \|_{\psi_{1,1}}\]
	%	and 
	then there holds the following relation between the norms induced by these two functions, for nonnegative scalar functions on $\cn$
	\begin{equation}\label{normequvalience2}
		\| \left( {\cdot}\right) ^{1/2} \|_{L_{\Phi}}^2 = \| \cdot \|_{L_{\Lambda}} \asymp \|\cdot \|_{\psi_{1,1}}
	\end{equation}
	which is readily observed from \eqref{Orlicznorm}. Although $\Lambda$ and $\psi_{1,1}$ are not Young functions, their associated norms are defined in the same form as Orlicz norms. The equivalence of the two norms on \eqref{normequvalience1} and \eqref{normequvalience2} will be used in the sequel without further mention. The constants in these comparisons are independent of $n$ and of the target Banach space; they depend only on the fixed functions defining the norms.

	\subsection{Mazur map}
	For $1\leq p,q<\infty$, the noncommutative Mazur map $\MM_{p,q}$ from $L_p(\M)$ to $L_q(\M)$ is defined by
	\[ \MM_{p,q}(x):=x|x|^{(p-q)/q}.\]
	Particularly, supposing $x=u|x|$ for $x \in L_{1}(\M)$ and $y \in L_{2}(\M)$, then we have
	\[{\MM}_{1,2}(x)=u|x|^{1/2},\qquad {\MM}_{2,1}(y)=y|y
	|.\]
	Moreover, this readily follows for every $x\in L_1(\M)$,
	\begin{equation}\label{mazur relations1}
		\MM_{2,1}(\MM_{1,2}(x))=x,
	\end{equation}
	
	\begin{equation}\label{mazur relations2}
		\|\MM_{1,2}(x)\|_{L_{2}(\M)}^2=\tau(|x|)=\|x\|_{L_{1}(\M)}.
	\end{equation}
	Ricard proved that ${\MM}_{p,q}$ is $\min \{p/q,1\}$-H\"older on the unit ball of $L_p(\M)$, with constants depending only on $p$ and $q$ (see \cite{R2015}). 
	
	\section{Related estimates and consequences}\label{esti}
	
	Firstly, for the left-hand side of the inequality \eqref{eqmain}, we will first carry out centering and then apply the Mazur map to perform the transformation. In our estimate of the left-hand side norm within the proof of Theorem \ref{main}, we first require an extension of a property of the Mazur map.
	\begin{theorem}[\cite{R2015}, page 1]\label{thm;maz}
		Let $\M$ be a von Neumann algebra, for $1\le p,q<\infty$, $\MM_{p,q}$ is $\min\left\{\frac{p}{q},1\right\}$-Hölder on the unit ball of $L_p(\mathcal{M})$.
	\end{theorem}
	
	This theorem describes the behavior of the map on the unit ball and we shall extend it to the whole noncommutative space.
	
	\begin{corollary}\label{cor:mazur}
		There exist universal constants $A,B<\infty$ such that for any $x,y\in L_1(\mathcal M)$ and any $u,v\in L_2(\mathcal M)$
		\begin{align}
			\left\|\MM_{1,2}(x)-\MM_{1,2}(y)\right\|_{L_{2}(\M)}
			&\leq A\|x-y\|_{L_{1}(\M)}^{1/2},\label{A}\\
			\left\|\MM_{2,1}(u)-\MM_{2,1}(v)\right\|_{L_{1}(\M)}
			&\leq B\max\{\left\|u\right\|_{L_{2}(\M)},\left\|v\right\|_{L_{2}(\M)}\}\,\left\|u-v \right\|_{L_{2}(\M)}.\label{B}
		\end{align}
	\end{corollary}
	
	\begin{proof}
		Ricard's first theorem in \cite{R2015} gives, on the unit ball of $L_{1}(\mathcal M)$,
		\[\|\MM_{1,2}(x)-\MM_{1,2}(y)\|_{L_{2}(\M)}\leq A\|x-y\|_{L_{1}(\M)}^{1/2}.\]
		For arbitrary $x,y$, let $r=\max\{\|x\|_{L_{1}(\M)},\|y\|_{L_{1}(\M)}\}$. When $r=0$, it is trivial. Since $\MM_{1,2}(rx)=r^{1/2}\MM_{1,2}(x)$ for $r\geq 0$, applying the unit-ball estimate to $x/r$ and $y/r$ gives \eqref{A} with the same constant.
		
		Likewise, Theorem \ref{thm;maz} says that $\MM_{2,1}$ is a Lipschitz operation on the unit ball of $L_{2}(\mathcal M)$. If $u=v=0$, \eqref{B} is immediate. For $s =\max\{\|u\|_{L_{2}(\M)},\|v\|_{L_{2}(\M)}\}>0$, we have
		\[\MM_{2,1}(su)=s^2 \MM_{2,1}(u).\]
		Then we can obtain \eqref{B} by
		\begin{align*}
			\|\MM_{2,1}(u)-\MM_{2,1}(v)\|_{L_{1}(\M)}
			&=s^2\left\|\MM_{2,1}\left(\frac{u}{s}\right)-\MM_{2,1}\left(\frac{v}{s}\right)\right\|_{L_{1}(\M)}\\
			&\leq Bs^2\left\|\frac{u-v}{s}\right\|_{L_{2}(\M)}\\
			&=Bs\|u-v\|_{L_{2}(\M)}.
		\end{align*}
	\end{proof} 
	It is worth noting that another important technical point in the proof consists in controlling the oscillation of the random variable by the partial derivatives of its components. Here we give a useful consequence that Cordero-Erausquin and Eskenazis have obtained in \cite{CE2023}:
	\begin{theorem}[Theorem 7, \cite{CE2023}]
		Let $(E,\|\cdot\|_E)$ be a Banach space with martingale type $2$.
		Then there exists $C(E)\in(0,\infty)$ such that for every
		$n\in\mathbb{N}$, every function $f:\cn\to E$ satisfies
		\begin{equation}
			\|f-{\E} f\|_{L_2(\cn;E)}^2 \leq C(E)\sum_{i=1}^n\|\partial_i f\|_{L_2(\log L)^{-1}(\cn;E)}^2.
		\end{equation}
	\end{theorem}
	\begin{theorem}\label{eq:cee-original}
		Let $H$ be an arbitrary Hilbert space. There is a constant $C_{\Xi}<\infty$ depending only on the choice of $\Xi$ in $\|\cdot\|_{L_{\Xi}}$ such that every $f:\cn\to H$ satisfies
		\begin{equation}
			\|f-{\E} f\|_{L_2(\cn;H)}^2 \leq C_{\Xi}\sum_{i=1}^n\|\partial_i f\|_{L_{\Xi}}^2.
		\end{equation}
		where $\Xi$ may be taken to be any Young function equivalent to $ \frac{t^2}{\log(e+t)}$, in particular the function $\Phi$, after changing the universal constant. 
	\end{theorem}
	\begin{remark}
		The constant can be chosen uniformly over all Hilbert spaces since the image of a map from the discrete cube spans a finite-dimensional Hilbert space, which embeds isometrically into a fixed $\ell_2$, so the constant for $\ell_2$ works for every such map. Here Theorem 3.3 applies to $\ell_2$ because it has martingale type $2$; the comparison with the chosen $L_\Xi$ norm contributes a constant depending only on $\Xi$.
	\end{remark}
	Equipped with all the above established results, we turn to the proof of our main theorem.
	
	\section{The proof of Theorem \ref{main}}\label{thmproof}
	\begin{proof}
		For any $f: \cn \to L_{1}\left(\M \right)$ with $p\geq1$, $f(\varepsilon)$ is a $L_{1}\left(\M \right)$-valued random variable. Supposing that $\delta$ is an independent copy of $\varepsilon$, then $f(\delta)$ and $f(\varepsilon)$ are i.i.d. We write $f(\varepsilon)$ briefly as $\x$ and $f(\delta)$ as $\x'$.
		Now we derive a bound for the left-hand side of \eqref{eqmain} via the Mazur map.
		Denote the centered variable by
		\[\y=\x-\E \x,\quad \y'=\x'-\E \x',\]
		and those under the $\{1,2\}$ Mazur map by
		\[G=\MM_{1,2}(\y), \quad G'=\MM_{1,2}(\y').\]
		Since $\E\y=\E\y'=0$ and centering does not change pairwise differences, by Jensen\'s inequality we have
		\begin{equation}\label{eq y}
			\E_{\y}\left\| \y\right\|_{L_{1}(\M)} = \E_{\y}\left\| \y-\E\y'\right\|_{L_{1}(\M)} \leq \E_{\y,\y'}\left\| \y-\y'\right\|_{L_{1}(\M)}=\E_{\x,\x'}\|\x-\x'\|_{L_{1}(\M)}.
		\end{equation}
		In the following, if two random variables are present, $\E$ denotes expectation over all random variables unless otherwise indicated by subscripts.
		
		Recall that we know \eqref{mazur relations2}, then
		\[\E\|G\|_{L_{2}(\M)}^2=\mathbb E\|\y\|_{L_{1}(\M)}.\]
		Moreover, invoking \eqref{mazur relations1} and \eqref{B}, it holds
		\begin{align*}
			\E\|\x-\x'\|_{L_{1}(\M)} 
			&= \E\|\y-\y'\|_{L_{1}(\M)}\\
			&=\E\|\MM_{2,1}(\MM_{1,2}(\y))-\MM_{2,1}(\MM_{1,2}(\y'))\|_{L_{1}(\M)}\\
			&= \E\|\MM_{2,1}(G)-\MM_{2,1}(G')\|_{L_{1}(\M)}\\
			&\leq B\,\E\left[\max\{\|G\|_{L_{2}(\M)},\|G'\|_{L_{2}(\M)}\}\,\|G-G'\|_{L_{2}(\M)}\right].
		\end{align*}
		The Cauchy–Schwarz inequality allows us to further estimate the above inequality
		\begin{align*}
			\E\|\y-\y'\|_{L_{1}(\M)}
			&\leq B\left(\E\max\{\|G\|_{L_{2}(\M)}^2,\|G'\|_{L_{2}(\M)}^2\}\right)^{1/2} \left(\E \|G-G'\|_{L_{2}(\M)}^2\right) ^{1/2}\\
			&\leq B(2 \E \|G\|_{L_{2}(\M)}^2)^{1/2} \left(\E \|G-G'\|_{L_{2}(\M)}^2\right)^{1/2}\\
			&= B(2 \E \|\y\|_{L_{1}(\M)})^{1/2} \left(\E \|G-G'\|_{L_{2}(\M)}^2\right) ^{1/2}.
		\end{align*}
		Since $L_2(\M)$ is a Hilbert space, the standard independent-copy identity gives
		\[\E \|G-G'\|_{L_{2}(\M)}^2=2 \E\|G-\E G\|_{L_{2}(\M)}^2.\]
		Using all the preceding derivations, \eqref{eq y} yields
		\begin{equation}\label{proof1}
			\begin{aligned}
				\E\|\x-\x'\|_{L_{1}(\M)}
				&\leq 4B^2 \E\|G-\E G\|_{L_{2}(\M)}^2\\
				&=4B^2\E \|\MM_{1,2}(\y(\varepsilon))-\E \MM_{1,2}(\y)\|_{L_{2}(\M)}^2.
			\end{aligned}
		\end{equation}
		
		We next reduce the problem to the sum of the norms of the partial derivatives to proceed with the proof. Theorem \ref{eq:cee-original} gives the desired bound
		\begin{equation}\label{proof2}
			\begin{aligned}
				\E \|\MM_{1,2}(\y)-\E \MM_{1,2}(\y)\|_{L_{2}(\M)}^2
				&\leq C_{\Phi}\sum_{j=1}^n\left\| \partial_j (\MM_{1,2} \circ \y)\right\|_{L_\Phi}^2\\
				&= C_{\Phi}\sum_{j=1}^n\left\|\,\|\partial_j (\MM_{1,2} \circ \y)\|_{L_{2}(\M)}\,\right\|_{L_\Phi}^2
			\end{aligned}
		\end{equation}
		as $L_2(\mathcal{M})$ is a Hilbert space.
		According to the definition of the derivation, 
		\begin{align}
			\| \partial_j \left( \MM_{1,2} \circ \y\right) (\varepsilon)\|_{L_{2}(\M)}
			&=\frac{1}{2}\|\MM_{1,2}(\y(\varepsilon))-\MM_{1,2}(\y(\varepsilon^{(j)}))\|_{L_{2}(\M)}\notag\\
			&\leq\frac A2\|\y(\varepsilon)-\y(\varepsilon^{(j)})\|_{L_{1}(\M)}^{1/2}\notag =\frac A{\sqrt2}\|\partial_j f(\varepsilon)\|_{L_{1}(\M)}^{1/2}.
		\end{align}
		where the first inequality follows from \eqref{A}. 
		By monotonicity and homogeneity of the Orlicz norm, followed by the identity and comparison in \eqref{normequvalience2},
		\begin{equation}\label{proof3}
			\begin{aligned}
				\bigl\|\,\|\partial_j (\MM_{1,2} \circ \y)\|_{L_{2}(\M)}\,\bigr\|_{L_\Phi}^2
				&\leq\frac{A^2}2\bigl\|\|\partial_j f(\varepsilon)\|_{L_{1}(\M)}^{1/2}\bigr\|_{L_\Phi}^2\\
				&\asymp\frac{A^2}2\bigl\|\|\partial_j f(\varepsilon)\|_{L_{1}(\M)}\bigr\|_{\psi_{1,1}}.
			\end{aligned}
		\end{equation}
		Combining \eqref{proof1}, \eqref{proof2} and \eqref{proof3}, we get
		\[\E \left\|f(\varepsilon)-f(\delta)\right\|_{L_{1}(\M)} \lesssim \sum_{j=1}^{n} \bigl\|\|\partial_j f(\varepsilon)\|_{L_{1}(\M)}\bigr\|_{\psi_{1,1}}\]
		Hence we complete the whole proof.
	\end{proof}
	\section*{Declaration on the use of AI}
	
	The authors used OpenAI Codex to assist with literature
	searches, identifying gaps, organizing and typesetting the manuscript, and editing the English. The authors take full responsibility for the content of this paper.


\begin{thebibliography}{10}

\bibitem{CE2023}
D.~Cordero-Erausquin and A.~Eskenazis.
\newblock Talagrand's influence inequality revisited.
\newblock {\em Anal. PDE}, 16(2):571--612, 2023.

\bibitem{Enflo1978}
P.~Enflo.
\newblock On infinite-dimensional topological groups.
\newblock In {\em S\'eminaire sur la {G}\'eom\'etrie des {E}spaces de {B}anach
  (1977--1978)}, pages Exp. No. 10--11, 11. \'Ecole Polytech., Palaiseau, 1978.

\bibitem{Enflo1969}
Per Enflo.
\newblock On the nonexistence of uniform homeomorphisms between
  {$L\sb{p}$}-spaces.
\newblock {\em Ark. Mat.}, 8:103--105, 1969.

\bibitem{HN2013}
Tuomas Hyt\"onen and Assaf Naor.
\newblock Pisier's inequality revisited.
\newblock {\em Studia Math.}, 215(3):221--235, 2013.

\bibitem{IVHV2020}
R.~Ivanisvili, P. an van~Handel and A.~Volberg.
\newblock Rademacher type and {E}nflo type coincide.
\newblock {\em Ann. of Math. (2)}, 192(2):665--678, 2020.

\bibitem{Ka1967}
M.~I. Kadec.
\newblock A proof of the topological equivalence of all separable
  infinite-dimensional {B}anach spaces.
\newblock {\em Funkcional. Anal. i Prilo\v zen.}, 1:61--70, 1967.

\bibitem{KKL1988}
J.~Kahn, G.~Kalai, and N.~Linial.
\newblock The influence of variables on {B}oolean functions.
\newblock In {\em 29th {A}nnual {S}ymposium on {F}oundations of {C}omputer
  {S}cience}, pages 68--80. IEEE Comput. Soc. Press, Washington, DC, [1988]
  \copyright 1988.

\bibitem{Ka2018}
G.~Kalai.
\newblock {B}oolean functions: influence, threshold and noise.
\newblock In {\em European {C}ongress of {M}athematics}, pages 85--110. Eur.
  Math. Soc., Z\"{u}rich, 2018.

\bibitem{Pi1986}
G.~Pisier.
\newblock Probabilistic methods in the geometry of {B}anach spaces.
\newblock In {\em Probability and analysis ({V}arenna, 1985)}, volume 1206 of
  {\em Lecture Notes in Math.}, pages 167--241. Springer, Berlin, 1986.

\bibitem{R2015}
É. Ricard.
\newblock H\"older estimates for the noncommutative {M}azur maps.
\newblock {\em Arch. Math. (Basel)}, 104(1):37--45, 2015.

\bibitem{MU1932}
Mazur S. and Ulam S.
\newblock Sur les transformations isométriques d’espaces vectoriels normés.
\newblock volume Vol.194, page pp.946–948. 1932.

\bibitem{Ta1994}
M.~Talagrand.
\newblock On {R}usso's approximate zero-one law.
\newblock {\em Ann. Probab.}, 22(3):1576--1587, 1994.

\end{thebibliography}
\end{document}